\documentclass[11pt]{amsart}
\usepackage{amssymb}
\usepackage[margin=1in]{geometry}
\usepackage[colorlinks]{hyperref}
\newtheorem{theorem}{Theorem}[section]

\newtheorem{cor}[theorem]{Corollary}

\theoremstyle{definition}
\newtheorem{definition}[theorem]{Definition}

\theoremstyle{remark}

\numberwithin{equation}{section}
\begin{document}

\newcommand{\spacing}[1]{\renewcommand{\baselinestretch}{#1}\large\normalsize}
\spacing{1.14}

\title{On the left invariant $(\alpha,\beta)$-metrics on some Lie groups}

\author {M. Hosseini}

\address{Department of Mathematics\\ Faculty of  Sciences\\ University of Isfahan\\ Isfahan\\ 81746-73441-Iran.} \email{hoseini\_masomeh@ymail.com}

\author {H. R. Salimi Moghaddam}

\address{Department of Mathematics\\ Faculty of  Sciences\\ University of Isfahan\\ Isfahan\\ 81746-73441-Iran.} \email{hr.salimi@sci.ui.ac.ir and salimi.moghaddam@gmail.com}



\begin{abstract}
We give the explicit formulas of the flag curvatures of left invariant Matsumoto and Kropina metrics of Berwald type. We can see these formulas are different from previous results given recently. Using these formulas, we prove that at any point of an arbitrary connected non-commutative nilpotent Lie group, the flag curvature of any left invariant Matsumoto and Kropina metrics of Berwald type admits zero, positive and negative values, this is a generalization of Wolf's theorem. Then we study $(\alpha,\beta)$-metrics of Berwald type and also Randers metrics of Douglas type on two interesting families of Lie groups considered by Milnor and Kaiser, containing Heisenberg Lie groups. On these spaces, we present some necessary and sufficient conditions for $(\alpha,\beta)$-metrics to be of Berwald type and also some necessary and sufficient conditions for Randers metrics to be of Douglas type. All left invariant non-Berwaldian Randers metrics of Douglas type are given and the flag curvatures are computed.
\end{abstract}

\maketitle

\section{\textbf{Introduction}}
The family of $(\alpha,\beta)$-metrics is an important and interesting class of Finsler metrics. These metrics are computable because they are expressed in terms of a Riemannian metric and a $1$-form. On the other hand such Finsler metrics have many applications in physics. In fact the first $(\alpha,\beta)$-metric, before it is named $(\alpha,\beta)$-metrics by M. Matsumoto in 1972 (see \cite{Matsumoto1}), was introduced by G. Randers in 1941 for its application in general relativity (see \cite{Randers}). This metric was named Randers metric by R. S. Ingarden \cite{Ingarden} and found many applications in physics. For example the authors of \cite{Chang-Li} have proposed a modified Friedmann model in Randers space for
possible alternative to the dark energy hypothesis. One can find another application of Finsler metrics in \cite{Li-Wang-Chang}, where the authors study the gravitational field equation in Randers space-time, and give a perturbational version of it (For more applications see \cite{AnInMa} and \cite{Asanov}).\\
A very useful and applicable special case happens when the invariant Finsler metrics on Lie groups are considered. In this case we can study the interactions between geometric properties and algebraic structures of Lie groups (as a comprehensive reference one can see \cite{Deng-book}). In this paper we consider left invariant $(\alpha,\beta)$-metrics on Lie groups. Computing flag curvature of Finsler metrics is an important and difficult problem in Finsler geometry. In the first part of this work, in the special cases of left invariant Matsumoto and Kropina metrics of Berwald type, the explicit formulas for computing flag curvatures are given. Easily we can see that these formulas are different from previous results recently given in \cite{Liu-Deng}. In \cite{Wolf}, J. Wolf showed that the sectional curvature of a left invariant Riemannian metric on an arbitrary connected non-commutative nilpotent Lie group admits zero, negative and positive values. This result is extended to Randers metrics in \cite{Deng-Hu Advances}. Using flag curvature formulas given in this article, we have generalized Wolf's result to left invariant Matsumoto and Kropina metrics of Berwald type on connected non-commutative nilpotent Lie groups. \\
An interesting family of Lie groups, which we have denoted by $\mathcal{G}_1$, is a particular class of solvable Lie groups introduced by J. Milnor in \cite{Milnor}. Let $G$ be a non-commutative Lie group and $\frak{g}$ be its Lie algebra. Then $G$ belongs to $\mathcal{G}_1$ if the Lie bracket of any two vectors $x, y\in \frak{g}$ is a linear combination of $x$ and $y$. In \cite{Milnor}, it is shown that if $G\in\mathcal{G}_1$ then there exist an element $b\in\frak{g}$ and a commutative ideal $\frak{u}$ of codimension one such that $[b,z]=z$ for any $z\in\frak{u}$.\\
 V. Kaiser considered another attractive class of Lie groups, the family of Lie groups with one-dimensional commutator groups. He studied the Ricci curvature of left invariant Riemannian metrics on these Lie groups (see \cite{Kaiser}). We have denoted this family by $\mathcal{G}_2$. We study $(\alpha,\beta)$-metrics of Berwald type and also Randers metrics of Douglas type on these two families of Lie groups. We give some necessary and sufficient conditions for $(\alpha,\beta)$-metrics to be of Berwald type and also some necessary and sufficient conditions for Randers metrics to be of Douglas type. On these spaces, we will present all left invariant non-Berwaldian Randers metrics of Douglas type. At any case, the flag curvature is computed.


\section{\textbf{Preliminaries}}
In this section we review some preliminaries about Riemannian and Finsler manifolds.\\
A Riemannian metric $g$ on a Lie group $G$ is called left invariant if
\begin{equation}
g(x)(y,z)=g(e)(dl\,_{x^{-1}} y,dl\,_{x^{-1 }} z) \quad    \text{for every}\quad         x\in G\quad ,\quad y,z\in T_x G,
\end{equation}
where $e$ denotes identity element.\\
Suppose that $G$ is a Lie group equipped with a left invariant Riemannian metric $g$, and $\frak{g}$ denotes its Lie algebra. Then for any two left invariant vector fields $X,Y\in\frak{g}$, the Levi-Civita connection and the sectional curvature can be computed as follows (see equations 2.1 and 2.2 of \cite{Alekseevskii}):
\begin{equation}\label{Levi-Civita connection}
\nabla _X Y=\dfrac{1}{2} (ad_X Y-ad_X^* Y-ad_Y^* X),
\end{equation}

\begin{equation}\label{sectional curvature}
K(X,Y)=\Vert \nabla _X Y\Vert ^2-g(\nabla _X X,\nabla _Y Y)-g(\left[ Y,\left[ Y,X\right] \right] ,X) -\Vert [X,Y]\Vert ^2,
\end{equation}
where $\Vert.\Vert^2:=g(.,.)$.
\begin{definition}
Let $M$ be differentiable manifold and $TM$ denotes its tangent bundle. Then a function $F:TM \longrightarrow \left[ 0, \infty \right)$ is said to be Finsler metric on the manifold $M$ if
\begin{itemize}
\item[i)] $F$ be a smooth function on $TM\backslash \{0\}$,
\item[ii)] $F(x,\lambda y)=\lambda F(x,y)$, \ \ \ \ for all $\quad  \lambda>0$,
\item[iii)] The hessian matrix $(g_{ij})=\left( \dfrac{1}{2} \dfrac{\partial ^2 F^2}{\partial y^i \partial y^j}\right) $ be positive definite for all $(x,y) \in TM\backslash \{0\}$.
\end{itemize}
\end{definition}
Similar to left invariant Riemannian metrics, a Finsler metric $F$ on a Lie group $G$ is called left invariant if
\begin{equation}
F(x,y)=F(e,dl\,_{x^{-1}}y)\ \ \text{for every}\ x\in G, y\in T_x G.
\end{equation}
An important family of Finsler metrics introduced by M. Matsumoto \cite{Matsumoto1}, is the class of $(\alpha , \beta)$-metrics. A Finsler metric is called an $(\alpha , \beta)$-metric if $F = \alpha \phi (\frac{\beta}{\alpha})$, where  $\alpha (x,y)=\sqrt{g_{ij}y^iy^j}$ and  $\beta (x,y)=b_iy^i$, and $g$ and $\beta$ are a Riemannian metric and a 1-form respectively, as follows
\begin{eqnarray}
 && g=g_{ij }dx^i \otimes  dx^j, \\
 &&\beta =b_i dx^i.
\end{eqnarray}
It can be shown that $F$ is a Finsler metric if $\phi : (-b_0 , b_0) \longrightarrow \mathbb{R^+}$ is a $C^\infty$ function satisfying
\begin{equation}
\phi (s) - s \phi ^{'} (s) + ( b^2 - s^2 ) \phi ^{''} (s) > 0, \qquad  \vert s \vert  \leq b < b_0,
\end{equation}
and  $\Vert \beta _x\Vert _{\alpha}=\sqrt{g^{ij} (x)b_i(x) b_j(x) } < b_0$ for any $x \in M$, where  $(g^{ij}(x))$ is the inverse matrix of $(g_{ij}(x))$ (for more details see  \cite{Chern-Shen}).\\
For instance if $\phi (s) = 1+s$, $\phi (s) = \frac{1}{1-s}$ or $\phi (s) = \frac{1}{s}$, then we obtain three important classes of Finsler metrics,  namely  Randers metrics $F=\alpha + \beta $, Matsumoto metrics $F=\frac{\alpha ^2}{\alpha - \beta} $ and Kropina metrics $F=\frac{\alpha ^2}{\beta}$, respectively.\\
In special case, it has been shown that a Randers metric $\alpha + \beta$ (Matsumoto metric  $\dfrac{\alpha ^2}{\alpha - \beta} $ ) is a Finsler metric if  and only if  $\Vert \beta _x\Vert _{\alpha}=\sqrt{g^{ij} (x)b_i(x) b_j(x) } < 1$ ($\Vert \beta _x\Vert _{\alpha}=\sqrt{g^{ij} (x)b_i(x) b_j(x) } < \frac{1}{2}$) (see \cite{Chern-Shen}).\\
In a natural way, the Riemannian metric $g$ induces an inner product on any cotangent
space $T^*_xM$. The induced inner product on $T^*_xM$ induces a linear isomorphism between $T^*_xM$  and $T_xM$ (for more details see \cite{Deng-Hou Phys. A}). Then the 1-form $\beta$ corresponds to a vector field $X$ on $M$ such that
\begin{equation}
g(y,X(x)) = \beta (x,y).
\end{equation}
Therefore we can consider Randers, Matsumoto and Kropina metrics as follows
\begin{equation}\label{Randers}
F(x,y)=\sqrt{g\left( y ,y \right)}+g\left( X\left( x\right) ,y\right) ,
\end{equation}
\begin{equation}\label{Matsumoto}
F(x,y)=\dfrac{g\left( y ,y\right)}{\sqrt{g\left( y ,y\right)}-g\left( X\left( x\right) ,y\right)},
\end{equation}
\begin{equation}\label{Kropina}
F(x,y)=\dfrac{g\left( y ,y\right)}{g\left( X\left( x\right) ,y\right)}.
\end{equation}
An important concept in Finsler geometry is flag curvature. This is a generalization of the concept of sectional
curvature in Riemannian geometry and defined by
\begin{equation}\label{flag curvature main formula}
K(P,y)=\dfrac{g_y \left( R_y \left( u\right) ,u\right) }{g_y (y,y) g_y (u,u)-g_y^2 (y,u)},
\end{equation}
where $P=\textit{span} \lbrace u,y \rbrace $, $g_y(u,v)=\frac{1}{2}\frac{\partial ^2}{\partial s \partial t}F^2(y+su+tv)\mid _{s=t=0}$, $R_y(u)=R(u,y)y=\nabla _u \nabla _y y -\nabla _y \nabla _u y- \nabla _{[u,y]}y$ and $\nabla$ is the Chern connection of $F$ (for more details see \cite{Bao-Chern-Shen,Chern-Shen}).\\
For a Finsler manifold $(M,F)$, in a standard local coordinate system of $TM$, the spray coefficients $G^i$ of $F$ are defined as follows:
\begin{equation}\label{spray coefficients}
G^i:=\frac{1}{4}g^{il}([F^2]_{x^my^l}y^m-[F^2]_{x^l}).
\end{equation}
The geodesics are characterized by the following system of ODEs:
\begin{equation}\label{Geodesic ODE}
    \ddot{x}^i+2G^i(x,\dot{x})=0.
\end{equation}
A Finsler metric $F$ is called a Berwald metric (or of Berwald type) if in any standard local coordinate system $(x,y)$ in $TM\backslash \{0\}$,
the Christoffel symbols of the Chern connection $\Gamma^i_{jk}$ have no dependence on $y$, in which case, the
Chern connection's Christoffel symbols $\Gamma^i_{jk}$ are functions only of $x\in M$, and  also the spray coefficients $G^i=\frac{1}{2}\Gamma^i_{jk}(x)y^jy^k$ are quadratic in $y$.\\
A Finsler metric $F$ is said to be of Douglas type (or a Douglas metric) if in any standard local coordinate system $(x,y)$ of $TM$, the spray coefficients $G^i$ of $F$ are of the form
\begin{equation}\label{spray coefficients of Douglas metrics}
    G^i=\frac{1}{2}\Gamma^i{jk}(x)y^jy^k+P(x,y)y^i,
\end{equation}
where $P(x ,y)$ is a local positively homogeneous function of degree one on $TM$ (see \cite{Bacso-Matsumoto} and \cite{Chern-Shen}). \\
For an $(\alpha,\beta)$-metric $F$ it has been shown that it is of Berwald type if and only if the 1-form $\beta$ is parallel with respect to the Levi-Civita connection of $\alpha$ (see \cite{Kikuchi} and \cite{Matsumoto2}). For Randers metrics in \cite{Bacso-Matsumoto} it is shown that a Randers metric is of Douglas type if and only if the 1-form $\beta$ is closed.

\section{\textbf{Flag curvature of left invariant Matsumoto and Kropina metrics of Berwald type}}
In this section we give the flag curvature formulas of left invariant Matsumoto and Kropina metrics of Berwald type on an arbitrary Lie group. In \cite{Deng-Hou Canad. J}, the authors have given the flag curvature formula of left invariant Randers metrics of Berwald type as follows:
\begin{equation}\label{flag curvature of Randers metrics of Berwald type}
K^F(P,y)=\dfrac{g^2}{F^2} K^g(P),
\end{equation}
for the flag $(P,y)$ and where $K^F$ and $K^g$ denote the flag curvature of $F$ and the sectional curvature of $g$, respectively. In \cite{Liu-Deng}, the authors have mentioned that the previous formula will remained valid for any left invariant $(\alpha,\beta)$-metric of Berwald type. But it seems that it is not true, because the fundamental tensors of $(\alpha,\beta)$-metrics are very different from each other.  The formulas given in this section show that the formula \ref{flag curvature of Randers metrics of Berwald type} is not true at least for left invariant Matsumoto and Kropina metrics of Berwald type. Therefore, the results of formula 4.6 of \cite{Liu-Deng} may need improvement.\\
After computing the flag curvatures, we generalize Wolf's result from left invariant Riemannian metrics to left invariant Matsumoto and Kropina metrics of Berwald type.\\

The following theorem shows that why we only consider the Berwaldian  Matsumoto and Kropina metrics.
\begin{theorem}
There isn't any non-Berwaldian Matsumoto and Kropina metrics of Douglas type.
\end{theorem}
\begin{proof}
Based on theorem 1.1 of \cite{Liu-Deng}, a left invariant Douglas $(\alpha,\beta)$-metric must be of Berwald type or of Randers type. Easily we can see Matsumoto and Kropina metrics aren't of Randers type.
\end{proof}
Now we compute the flag curvatures.\\
\begin{theorem}
Let $G$ be a Lie group and $\frak{g}$ denotes its Lie algebra. Assume that $F$ is a left invariant Matsumoto
metric of Berwald type, arisen from a left invariant Riemannian metric $g$ and a left invariant vector field $X$. Also let $(P,y)$ be a flag in $\frak{g}$ such that $\{v,y\}$ be an orthonormal basis of $P$ with respect to $g$. Then the flag curvature $K^F(P,y)$ satisfies the following formula
\begin{equation}\label{flag curvature of Matsumoto metric of Berwald type}
  K^F(P,y) =\frac{2-F}{F^2(2F^2g^2(X,v)+2-F)} K^g (P),
\end{equation}
where $K^g$ denotes the sectional curvature of the Riemannian metric $g$.
\end{theorem}
\begin{proof}
Since $F$ is of Berwald type, the Levi-Civita connection of $g$ and the Chern connection of $F$ coincide. Now by using \ref{Levi-Civita connection},  formula 2.4 of \cite{Salimi Osaka} and theorem \ref{Berwald conditions}, for $g_y(R_y(v),v)$ we have
\begin{align*}
g_y \left( R_y \left( v\right) ,v\right) =&\left(\dfrac{2}{(1-g(X,y))^2}-\dfrac{1}{(1-g(X,y))^4}+\dfrac{g(X,y)}{(1-g(X,y))^4}\right) g\left( R_y \left( v\right) ,v\right)\\
=& (2F^2-F^3)K^g (P).
\end{align*}
Also formula 2.6 of \cite{Salimi Osaka} shows that
\begin{align*}
g_y (y,y) g_y (v,v)-g_y^2 (v,y)=&\dfrac{2}{(1-g(X,y))^4}+\dfrac{2g^2 (X,v)+g(X,y)-1}{(1-g(X,y))^6}\\
=&2F^4+2F^6g^2(X,v)-F^5.
\end{align*}
Now, we can complete the proof by using the relation \ref{sectional curvature} and \ref{flag curvature main formula}.
\end{proof}

\begin{theorem}
In the previous theorem if we replace Matsumoto metrics by Kropina metrics and if we suppose that the Chern connection of $F$ coincide to the Levi-Civita connection of $g$, then the flag curvature $K^F(P,y)$ satisfies the following formula
\begin{align}\label{flag curvature of Kropina metric of Berwald type}
  K^F(P,y) =\dfrac{1}{F^4 g^2\left( X,v \right) +F^2 } K^g (P).
\end{align}
\end{theorem}
\begin{proof}
The proof is similar to the previous theorem. It is sufficient to use formula 2.5 of \cite{Salimi Electronic}, then we have
\begin{eqnarray*}
  g_y\left(R_y(v),v\right) &=& \dfrac{2 g\left(R_y(v),v\right)}{g^2\left( X,y \right) }\\
    &=&2F^2K^g(P).
\end{eqnarray*}
On the other hand by using formula 2.7 of \cite{Salimi Electronic} we have
\begin{align*}
g_y (y,y) g_y (v,v)-g_y^2 (v,y)=&\frac{2g^2(X,v)}{g^6(X,y)}+\frac{2}{g^4(X,y)}\\
=&2F^4(F^2g^2(X,v)+1).
\end{align*}
The above formulas complete the proof.
\end{proof}
J. Wolf, in \cite{Wolf},  proved that the sectional curvature of a left invariant Riemannian metric on an arbitrary connected non-commutative nilpotent Lie group admits zero, negative and positive values. Now by using the above two theorems we generalize Wolf's theorem to Matsumoto and Kropina metrics of Bervald type.
\begin{theorem}
Let $G$ be a connected non-commutative nilpotent Lie group equipped with a left invariant Matsumoto (Kropina) metric $F$, arising from a left invariant Riemannian metric $g$ and a left invariant vector field $X$ such that the Chern connection of $F$ coincides to the Levi-Civita connection of $g$. Then, at any point $x\in G$, the flag curvature admits positive, negative and zero values.
\end{theorem}
\begin{proof}
Regarding \cite{Wolf}, it is sufficient to show that the coefficients of $K^g$ in the formulas of flag curvatures given in the above theorems have a constant sign. Obviously, the coefficient of $K^g$ in \ref{flag curvature of Kropina metric of Berwald type} is always positive, so we only prove the case of Matsumoto metric. In the previous theorems we have considered that $g(y,y)=1$, so we have $F(y)=\frac{1}{1-g(X,y)}$. On the other hand $F$ is a Matsumoto metric hence $\sqrt{g(X,X)}<\frac{1}{2}$. Now the Cauchy-Schwarz inequality shows that $|g(X,y)|<\frac{1}{2}$, which means that $0<2-F(y)<\frac{4}{3}$. Therefore the coefficient of $K^g$ in the formula \ref{flag curvature of Matsumoto metric of Berwald type} is also positive.
\end{proof}

\section{\textbf{Left invariant $(\alpha,\beta)$-metrics of Berwald type and Randers metrics of Douglas type on Lie groups of type $\mathcal{G}_1$}}
In this section, firstly we give a necessary and sufficient condition for a left invariant $(\alpha,\beta)$-metric on an arbitrary Lie group to be of Berwald type. In fact we have shown that the criterion given in theorem 3.1 of \cite{An-Deng Monatsh}, for left invariant Randers metrics, can be extended to left invariant $(\alpha,\beta)$-metrics. Then we study the Lie group family ${\mathcal{G}}_1$ equipped with a left invariant $(\alpha , \beta)$-metric.

\begin{theorem}\label{Berwald conditions}
Let $F$ be a left invariant $(\alpha,\beta)$-metric on a Lie group $G$, arising from a left invariant Riemannian metric $g$ and a left invariant vector field $X$. Then $F$ is of Berwald type if and only if the following two conditions are valid:
\begin{align}
&g\left( \left[ y,X\right] ,z \right) + g\left( \left[ z,X\right] ,y\right) =0, \qquad \forall \, y,z \in \mathfrak{g}, \label{Berwald condition 1} \\
&g\left( \left[ y,z\right] ,X\right)= 0, \qquad \forall \, y,z \in \mathfrak{g}. \label{Berwald condition 2}
\end{align}
\end{theorem}
\begin{proof}
The proof is similar to the proof of theorem 3.1 of \cite{An-Deng Monatsh}. By \cite{Kikuchi} and \cite{Matsumoto2} we know that the $(\alpha,\beta)$-metric $F$ is of Berwald type if and only if the vector field $X$ is parallel with respect to the Levi-Civita connection of $g$, and for left invariant Riemannian metrics, this is equivalent to the following equation
\begin{equation}\label{parallel condition for X}
    g\left( \left[ y,X\right] ,z \right) + g\left( \left[ z,X\right] ,y\right) +g\left( \left[ z,y\right] ,X\right)= 0, \qquad \forall \, y,z \in \mathfrak{g}.
\end{equation}
Now if we let $y=z$, then for every $y \in \mathfrak{g}$ we have $g\left( \left[ y,X\right] ,y \right)=0$. The last equation shows that for any $y,z \in \mathfrak{g}$, we have $g\left( \left[ y+z,X\right] ,y+z \right)=0$ which is equivalent to $g\left( \left[ y,X\right] ,z \right)+g\left( \left[ z,X\right] ,y \right)=0$. The last equation together with equation \ref{parallel condition for X} show that $g\left( \left[ y,z\right] ,X\right)= 0$, for any $y,z \in \mathfrak{g}$. The converse is clear.
\end{proof}

Suppose that $G\in{\mathcal{G}}_1$ and $\mathfrak{g}$ denotes its Lie algebra. Let $g$ be a left invariant Riemannian metric on $G$. As we mentioned in introduction, it has been shown that if $G\in\mathcal{G}_1$ then there exist an element $b\in\frak{g}$ and a commutative ideal $\frak{u}$ of codimension one such that $[b,z]=z$ for any $z\in\frak{u}$ (for more details see \cite{Milnor} and \cite{Nomizu}).\\
\begin{theorem}
Let $G \in {\mathcal{G}}_1 $ be a Lie group with Lie algebra $\mathfrak{g}$. Suppose that $F$ is a left invariant Randers metric defined by a left invariant Riemannian metric $g$ and a left invariant vector field $X$ on $G$. Then, $F$ is of Douglas type if and only if $X \in\emph{span}\{b\}$.
\end{theorem}
\begin{proof}
Theorem 3.2 of \cite{An-Deng Monatsh} says that $F$ is of Douglas type if and only  if $X$ is orthogonal to $[\mathfrak{g} , \mathfrak{g}]$. It holds if and only if $X \in\emph{span}\{b\}$.
\end{proof}

\begin{theorem}
Suppose that $G$ belongs to the family ${\mathcal{G}}_1$. Then, $G$ doesn't admit any left invariant non-Riemannian $(\alpha,\beta)$-metric of Berwald type which is defined by a left invariant Riemannian metric $g$ and a left invariant vector field $X$ .
\end{theorem}
\begin{proof}
Let $F$ be such an $(\alpha,\beta)$-metric. By theorem \ref{Berwald conditions} we have:
\begin{equation}
g\left( [z,y] , X \right) =0,\qquad \forall z,y \in \mathfrak{g}, \label{*}
\end{equation}
\begin{equation}
g\left( [X,z] , y \right) + g\left( [X,y] , z \right) =0,\qquad \forall z,y \in \mathfrak{g}. \label{**}
\end{equation}
Equality \ref{*} holds if and only if $X\in\emph{span}\{b\}$. Now if for $\xi\in\mathbb{R}$, we substitute $X=\xi b$ in formula \ref{**}, then we will have $X=0$.
\end{proof}
The above two theorems give us the following corollary.
\begin{cor}
Any Lie group $G\in{\mathcal{G}}_1$ admits left invariant non-Berwaldian Randers metrics of Douglas type.
\end{cor}
Now we present the flag curvature formula of such left invariant non-Berwaldian Randers metrics of Douglas type given in the previous corollary.

\begin{theorem}\label{flag curvature of Randers metrics of Douglas type on G_1}
Let $G\in{\mathcal{G}}_1$ be a Lie group with Lie algebra $\mathfrak{g}$ equipped with a Randers metric $F$ of Douglas type which is arisen from a left invariant Riemannian metric $g$ and a left invariant vector field $X$. Then the flag curvature of $F$ is given by
\begin{equation}
K^F(P,y)=\dfrac{g^2}{F^2} K^g(P)+\dfrac{1}{4F^4}\left( 3\xi ^2g^2\left(v,v \right)+4F\xi \eta g\left( v,v\right)   \right),
\end{equation}
where $K^g$ is the sectional curvature of $g$, $ X = \xi b$ ( $\xi \in \mathbb{R}$ ) and $y=\eta b+v$ ($\eta \in \mathbb{R}, v \in \mathfrak{u}$). In special cases we have:
\begin{align*}
&K^F(P,y) = \dfrac{g^2}{F^2}K^g(P), \qquad y\in span\{b\} \\
&K^F (P,y) = \dfrac{g^2}{F^2} K^g (P) + \dfrac{3}{4F^4} \xi ^2 g^2 \left( v,v \right) , \qquad  y \in \mathfrak{u}.
\end{align*}
\end{theorem}

\begin{proof}
We use the following formula of the flag curvature which is given in theorem 2.1 of \cite{Deng-Hou Canad. J}.
\begin{equation}\label{flag curvature of Randers metrics of Douglas type}
K^F(P,y)=\dfrac{g^2}{F^2} K^g(P)+\dfrac{1}{4F^4} \left( 3g^2\left( U\left( y,y\right),X \right)-4Fg\left( U\left( y,U\left( y,y\right) \right),X  \right)   \right),
\end{equation}
where $K^g$ denotes the sectional curvature of $g$, and $U:\mathfrak{g} \times \mathfrak{g} \longrightarrow \mathfrak{g}$ is the bilinear symmetric map defined by
\begin{equation*}
2g\left( U\left( w,y\right),z \right) = g\left( [z,w],y\right) + g\left( [z,y],w\right), \qquad \forall w, y, z \in \mathfrak{g}.
\end{equation*}
A direct computation shows that
\begin{equation}\label{***}
K^F(P,y)=\dfrac{g^2}{F^2} K^g(P) + \dfrac{1}{4F^4} \left( 3g^2\left( [X,y],y\right) - 2F \left( g\left( \left[ \left[ X,y\right],y \right],y \right) - g\left( y,\left[ X,ad^*_yy\right] \right)   \right) \right),
\end{equation}
where $ad^*_y$ denotes the transpose of $ad_y$ with respect to $g$. It is easy to show that
\begin{align*}
&g\left( [X,y],y\right) = \xi g\left(v,v \right), \\
&g\left( \left[ \left[ X,y\right],y \right],y \right) = -\xi \eta g\left( v,v\right), \\
&g\left( y,\left[ X,ad^*_yy\right] \right) =  \xi \eta g\left( v,v\right).
\end{align*}
Now substituting the above equations in \ref{***} completes the proof.
\end{proof}

\section{\textbf{Left invariant $(\alpha,\beta)$-metrics of Berwald type and Randers metrics of Douglas type on Lie groups of type $\mathcal{G}_2$}}
Let $G\in{\mathcal{G}}_2$ be a Lie group equipped with a left invariant Riemannian metric $g$, and $G^\prime$ denotes its one dimensional commutator group. Suppose that $\mathfrak{g}$ and $\mathfrak{g^\prime}$ are the Lie algebras of $G$ and $G^\prime$, respectively. Assume that $\textsf{e}\in\mathfrak{g^\prime}$ is a unit element and $\Gamma\subset\mathfrak{g}$ is the hyperplane which is orthogonal to the vector $\textsf{e}$ with respect to $g$. So there exits a linear map $\phi : \Gamma \longrightarrow \mathbb{R}$ such that $[z,\textsf{e}]=\phi(z)\textsf{e}$, for any $z\in\Gamma$. It is clear that there exits a unique vector $\textsf{a}\in\Gamma$, such that for any $z\in\Gamma$ we have $\phi(z)=g(\textsf{a},z)$. So we have
\begin{center}
$[z,\textsf{e}]=g(\textsf{a},z)\textsf{e},\qquad \forall z\in\Gamma$.
\end{center}
Let $B$ be the skew-symmetric bilinear form on $\Gamma $, such that $[z,y]=B(z,y)\textsf{e}$ for any $z$ and $y$ in $\Gamma $. Then there exists a unique skew-adjoint linear transformation $f : \Gamma \longrightarrow \Gamma$ such that $B(z,y)=g\left( f \left( z\right) ,y\right) $ for each $z$ and $y$ in $\Gamma$ (for more details see \cite{Kaiser}). Therefore we have
\begin{center}
$[z,y]=g\left( f \left( z\right) ,y\right)\textsf{e}, \qquad \forall  z,y\in\Gamma$.
\end{center}

\begin{theorem}
Let $G\in{\mathcal{G}}_2$ be a Lie group equipped with a left invariant $(\alpha ,\beta)$-metric $F$ defined  by a left invariant Riemannian metric $g$ and a left invariant vector field $X$. Then
\begin{itemize}
\item[a)] $F$ is a Randers metric of Douglas type if and only if $X\in \Gamma$,
\item[b)] $F$ is of Berwald type if and only if  $X\in \Gamma \cap \mathfrak{z}(\mathfrak{g})$,
\end{itemize}
where $\mathfrak{z}(\mathfrak{g})$ denotes the center of $\mathfrak{g}$.
\end{theorem}
\begin{proof}
a) According to the theorem 3.2 of \cite{An-Deng Monatsh}, it is sufficient to show that $X\in \Gamma$ if and only if $g\left([z,y],X \right)=0$ for all $z,y\in\mathfrak{g}$. But it is clear because $G$ is not commutative.\\
b) It is sufficient to use theorem \ref{Berwald conditions}. Similar to part (a) $X\in \Gamma$ if and only if $g\left([z,y],X \right)=0$ for all $z,y\in\mathfrak{g}$. On the other hand, for any $z,y\in\mathfrak{g}$, $g\left( [X,z] , y \right) + g\left( [X,y] , z \right) =0$ if and only if $f(X)=\phi(X)=0$ which is equivalent to $X\in\mathfrak{z}(\mathfrak{g})$.
\end{proof}
The family of Lie groups ${\mathcal{G}}_2$ is an important class of Lie groups because it contains the family of Heisenberg Lie groups.
Let $H_{2n+1}$ be a Heisenberg group with Lie algebra $\mathfrak{h}_{2n+1}$. Then there exists a basis $\{ x_1,y_1,x_2,y_2,...,x_n,y_n,z\}$ for $\mathfrak{h}_{2n+1}$ such that the none zero commutators are as follows (for more details see \cite{Vukmirovic}):
\begin{equation*}
[x_i,y_i]= - [y_i,x_i] =z, \qquad  \text{for} \, i=1,...,n.
\end{equation*}
So we can consider $\textsf{e}=z$ and $\Gamma = \textit{span}\{ x_1,y_1,x_2,y_2,...,x_n,y_n\}$.
\begin{cor}
The Heisenberg group $H_{2n+1}$ doesn't admit any non-Riemannian $(\alpha ,\beta)$-metrics of Berwald type.
\end{cor}
\begin{proof}
The intersection  of center of Lie algebra $\mathfrak{h} _{2n+1}$ and the hyperplane $\Gamma$ in the Heisenberg group is $\{0\}$. So, any $(\alpha , \beta)$-metrics of Berwald type must be Riemannian.
\end{proof}
Similar to the previous section we give the flag curvature formulas for left invariant Randers metric of Douglas type on such spaces.

\begin{theorem}\label{flag curvature of Randers metrics of Douglas type on G_2}
Suppose that $G\in{\mathcal{G}}_2$ is a Lie group equipped with a left invariant Randers metric of Douglas type which is defined by a left invariant Riemannian metric $g$ and a left invariant vector field $X$. Then the flag curvature is given by
\begin{align}
K^F(P,y) =& \dfrac{g^2}{F^2} K^g(P) + \dfrac{1}{4F^4}\lbrace  3( \eta ^2 \phi  (X)+ \eta g( f(X),v)) ^2\nonumber\\
&\hspace*{3cm}+2F( 2 \eta ^2\phi (v) \phi (X) +\eta \phi (v)g(f(X),v)\\
&\hspace*{3cm}+\eta ^2 g(f(X),f(v)) -\eta ^3 \phi( f(X))) \rbrace ,    \nonumber
\end{align}
where $K^g$ is the sectional curvature of $g$ and $y=\eta\textsf{e}+v$ ($\eta \in \mathbb{R}, v \in \Gamma$). Specially,
\begin{align*}
&K(P,y)=\dfrac{g^2}{F^2} K^g(P) +\dfrac{1}{4F^4} \left( 3\eta ^4 \phi ^2 (X) -2F \eta ^3 \phi (f(X))\right), \qquad y\in \textit{span}\{\textsf{e}\}, \\
&K(P,y)=\dfrac{g^2}{F^2} K^g(P), \qquad y \in \Gamma.
\end{align*}
\end{theorem}
\begin{proof}
Similar to the proof of the theorem \ref{flag curvature of Randers metrics of Douglas type on G_1}, it is sufficient to compute $g\left(\left[ X,y\right] ,y \right)$,  $g\left( \left[ \left[ X,y\right] ,y\right] ,y \right)$ and $g\left( y,\left[X, ad^*_yy \right] \right)$. A simple computation shows that
\begin{align}
&g\left( [X,y],y\right) =\eta ^2 \phi (X)+\eta g\left( f(X),v\right),  \label{a}  \\
&g\left( \left[ \left[ X,y\right],y \right],y \right) = - \phi (v)\left( \eta ^2 \phi (X)+\eta g\left( f(X),v\right) \right) ,    \label{b} \\
&g\left( y,\left[ X,ad^*_yy\right] \right) = \eta ^2 g\left( f(X),f(v)\right) - \eta ^3 \phi\left(f(X) \right) +\eta ^2 \phi (X) \phi (v).   \label{c}
\end{align}
Now the substitution of the equations \ref{a}, \ref{b} and \ref{c} in \ref{***} completes the proof.
\end{proof}
\begin{cor}
In the previous theorem if $\textsf{e}$  belongs to the center of $\mathfrak{g}$, then the flag curvature is given by
\begin{equation}
K^F(P,y) = \dfrac{g^2}{F^2} K^g(P) + \dfrac{\eta ^2}{4F^4} \left( 3g^2 \left( f(X),v\right) +2F  g \left( f(X),f(v)\right)  \right).
\end{equation}
A familiar case happens when we consider the Heisenberg groups.
\end{cor}
\begin{proof}
In this case we have $\phi \equiv 0$. So the proof is clear.
\end{proof}

\end{document}